\documentclass[12pt]{article}
\usepackage{amsmath}
\usepackage{amsfonts}
\usepackage{amssymb}
\usepackage{graphicx}
\newtheorem{theorem}{Theorem}[section]
\newtheorem{lemma}{Lemma}[section]
\newtheorem{corollary}{Corollary}[section]
\newtheorem{proposition}{Proposition}[section]

\newenvironment{proof}{\smallskip\noindent{\it Proof}}{$\Box$}
\numberwithin{equation}{section}
\newcommand{\blambda}{{\boldsymbol \lambda}}

\def\bh{{\boldsymbol h}}

\def\bC{{\boldsymbol C}}
\def\bD{{\boldsymbol D}}
\def\bSigma{{\boldsymbol \Sigma}}
\def\bx{{\boldsymbol x}}

\def\bv{{\boldsymbol v}}
\def\bq{{\boldsymbol q}}
\def\bp{{\boldsymbol p}}
\def\bd{{\boldsymbol d}}

\begin{document}

\title{Sample dependence in the maximum entropy solution to the generalized moment problem} 
\author{Henryk Gzyl\\
\noindent 
Centro de Finanzas IESA, Caracas, (Venezuela)\\
 henryk.gzyl@iesa.edu.ve\\
}

\date{}
 \maketitle

\begin{abstract}
The method of maximum entropy is quite a powerful tool to solve the generalized moment problem, which consists of determining the probability density of a random variable $X$ from the knowledge of the expected values of a few functions of the variable. In actual practice, such expected values are determined from empirical samples, leaving open the question of the dependence of the solution upon the sample. It is the purpose of this note to take a few steps towards the analysis of such dependence.
\end{abstract}

\noindent {\bf Keywords}:  Generalized moment problem. Maximum entropy method. Sample dependence. Entropy convergence.
\baselineskip=1.5 \baselineskip \setlength{\textwidth}{6in}

\section{Introduction and preliminaries}
To state what the generalized moment problem is about, let $(\Omega, \mathcal{F},\mathbb{P})$ be a probability space and let $(S, \mathcal{B}, m)$ be a measure space, with $m$ a finite or $sigma-$finite measure. Let $X$ be an $S-$valued random variable, such that its distribution has a density with respect to the measure $m.$  The generalized moment problem consists in determining a density $f(\bx)$ such that
\begin{equation}\label{ME}
\int_S h_k(x)f(x)m(dx) = d_k\qquad\mbox{for}\qquad k=0,1,...,M.
\end{equation}
\noindent
\noindent where $h_k$ is a collection of measurable functions $h_k:S\rightarrow\mathbb{R},$ the $d_k$ are given real numbers, and we set $h_0 \equiv 1$ and $d_0=1$ to take care of the natural requirement on $f.$  A typical example is the following. $X$ stands for a positive random variable (a stopping time, o perhaps a total risk severity) and we can compute $E[\exp(-\alpha_k X)] = d_k$ by some Montecarlo procedure at a finite number of points $\alpha_k.$ The problem that we need to solve amounts to invert the Laplace transform from such finite collection of values of the transform parameter $\alpha.$

Actually this last problem is of much interest in the banking and insurance industries, where the density is necessary to compute risk premia and regulatory capital of various types, samples may be small and the estimation of the $d_k$ reflects that. We mention the unpublished work by Leitner and Temnov  (2009) in this regard, or Gomes-Gon\'{c}alvez et al (2015), where this theme was addressed in the context of risk modeling and Laplace transform inversion.

Thus, to sum up, the problem that we address in this work is the analysis of the dependence of the $f^*,$ reconstructed with the aid of the maximum entropy method, on the sample used to estimate the $d_k'$s. To put it in symbols, if $X_1, ....,X_N$ is a sample of $X$ of size $N$ and we estimate the $d_k'$s by
\begin{equation}\label{estim1}
\hat{d}_k = \frac{1}{N}\sum_{n=1}^N h_k(X_n)
\end{equation}

\noindent we want to understand the variability in the density $\hat{f}^*_N$ obtained applying the maximum entropy method, due to the variability of the sample $X_1, ....,X_N.$ For that, in the next section we recall in a (short) historical survey the notion of entropy of a density, and in the following section we present the basics of the maximum entropy method.

In Section 4 we take up the main theme of this work: the variability of $\hat{f}^*_N$ that comes in through the $\hat{d}_k.$ There we prove that $\hat{f}^*_N$ converges pointwise and in $L_1$ to the maxentropic density $f^*$ obtained from the exact data, and we shall see examine how $\hat{f}^*_N$ deviates from $f^*$ in terms of the difference between the true and the estimated (sample) moments.

\section{The entropy of a density}
As there seem to be several notions of entropy, it is the aim of this section to point out that they are all variational on the theme of a single definition. 
Let us begin by spelling out what is it that we call the entropy of a density. Let $P$ be a measure on $(S, \mathcal{B}).$ Suppose that $P << m$ and let $f$ denote its density.  The entropy $S_m(P)$ is defined by
\begin{equation}\label{ent1}
S_m(P) = -\int_S f\ln(f)dm
\end{equation}
\noindent whenever $\ln(f)$ is $P-$integrable, or $-\infty$ if not. Actually, we can also define $S_m(f)$ for $f \in \{ g \in L_1(dm), g \geq 0\}.$

When $P$ is not a probability measure, (\ref{ent1})  is to be modified as follows:
\begin{equation}\label{ent2}
S_m(P) = S_M(f) = -\int_S f\ln(f)dm + (\int_S fdm - 1).
\end{equation}
When $m$ is a probability measure, call it $Q,$ and both $P$ and $Q$ are equivalent  to a measure $n,$ with densities, respectively $f=dP/dm$ and $g=dQ/dm,$ then  (\ref{ent1}) can be written as
\begin{equation}\label{ent3}
S_n(P,Q) = -\int_S \left(\frac{f}{g}\right)\ln\left(\frac{f}{g}\right)gdn = -\int_S\ln\left(\frac{f}{g}\right)fdn = -\int_S\ln\left(\frac{f}{g}\right)dP.
\end{equation}
{\bf Comment} For the applications that we shall be dealing with, $S$ will stand for a closed, convex subset in some $\mathbb{R}^n,$ and $m$ will be the usual Lebesgue measure.  We also mention that if $m$ is a discrete measure, then the integrals would become sums.

The expression (\ref{ent1}) seems to have made its first appearance in the work of Boltzmann in the last quarter of the XIX-th century. There it was defined in $\mathbb{R}^3\times\mathbb{R}^3,$ where $f(\bx,\bv)d\bx d\bv$ was to be interpreted as the number of particles with position within $d\bx$ and velocity within $\bv.$ The function happened to be a Lyapunov functional for the dynamics that Boltzmann proposed for the evolution of a gas, which grew as the gas evolved towards equilibrium. Not much later Gibbs used the same function, but now defined on $\mathbb{R}^{6N}\times\mathbb{R}^{6N},$ whose points $(\bq,\bp)$ denote the joint position and momenta of a system of $N$ particles. This time $dm=d\bq d\bp,$ and $f(\bq,\bp)d\bq d\bp$ is the probability of finding the system within the specified ``volume'' element. Motivated by earlier work in thermodynamics, it was postulated that in equilibrium the density of the system yielded a maximum value to the entropy $S_m(f).$ These remarks explain the name of the method.

The expression (\ref{ent1}) (with a reversed sign) made its appearance in field of information transmission under the name of information content in the density $f,$ that is why it is sometimes called the Shannon-Boltzmann entropy. Also, expression (\ref{ent3})  appeared in the statistical literature under the name of (Kullback-) divergence of the density $f$ with respect to the density $g,$ and denoted by $K(f,g)$ and equal to $-S_n(P,Q).$ See Cover and Thomas (2006) or Kullback (1968) for a detailed study of the properties of he entropy functions.

Having made those historical remarks, and having stated those equivalent definitions, we mention that we shall be working with mostly with (\ref{ent1}). In what comes below we make use of some interesting and well known properties of (\ref{ent1}) and  (\ref{ent3}) , which we gather under
\begin{theorem}\label{theo1}
With the notation introduced above\\
(i) The function $f \rightarrow S_m(f)$ is strictly concave.\\
(ii) For any two densities $f$ and $g,$ $S_n(f,g) \leq 0,$ and $S_n(f,g)=0$ if and only if $f=g$ a.e. $n.$\\
(iii) For any two densities $f,g$ such that $S_n(f,g)$ is finite, we have (Kullback's inequality)
$$\frac{1}{4}|\left(\|f-g\|_1\right)^2 \leq -S_n(f,g).$$
\end{theorem}
The reader is directed to either Cover and Thomas (2003) or to Kullback (1968) for proofs.
\section{The standard maximum entropy method}
Here we recall some well known results about the  standard maximum entropy method (SME) along with some historical remarks. Even though the core idea seems to have been first made in the work of Esscher (1932), where he introduced what nowadays is called the Esscher transform,  it was not until the mid 1950's that it became part of the methods used in statistics, through the work of Kullback (1968). It seems to have been first been formulated as a variational procedure by Jaynes (1957) to solve the (inverse) problem consisting of finding a probability density $fy)$ (on the phase space of a mechanical system), satisfying the following integral constraints:
\begin{equation}\label{ME}
\int_S h_k(x)f(x)dm(x) = d_k\qquad\mbox{for}\qquad k=0,1,...,M.
\end{equation}
\noindent
where the $d_k$ are observed (measured) expected values of some functions (``observables'', or random variables in the probabilistic terminology) of the system. That problems appears in many fields, see Kapur (1989) and Jaynes (2003) for example.

Usually, we set $h_0\equiv 1$ and $d_0=1$ to take care of the natural requirement on $f(x).$   It actually takes a standard computation to see that when the problem has a solution, it is of the type
\begin{equation}\label{sol1}
f_M^{*}(x) = \exp\left(-\sum_{k=0}^M \lambda^{*}_k h_k(x)\right)
\end{equation}
\noindent in which the number of moments $M$ appears explicitly. It is usually customary to write $e^{-\lambda^{*}_0} = Z(\mbox{\boldmath $\lambda$}^{*})^{-1},$ where $\mbox{\boldmath $\lambda^{*}$} = (\lambda_1^{*},...,\lambda_M^{*})$ is an $M-$dimensional vector.
Clearly, the generic form of the normalization factor is given by
\begin{equation}\label{zeta}
Z(\mbox{\boldmath $\lambda$}) = \int_0^1 e^{-\sum_{k=1}^M \lambda_k h_k(x)}m(dx).
\end{equation}
 With this notation the generic form of the solution can be rewritten as
\begin{equation}\label{sol2}
f_M^{*}(x) = \frac{1}{Z(\mbox{\boldmath $\lambda$}^{*})}e^{-\langle\lambda^{*}, \bh(x)\rangle} = e^{-\sum_{k=0}^M \lambda^{*}_k h_k(y)}.
\end{equation}
Here $\langle\mathbf{a},\mathbf{b}\rangle$ denotes the standard Euclidean scalar product in $\mathbb{R}^M,$ and $\bh(x)$ is the vector with components $h_k(x).$ At this point we mention that the simple minded proof appearing in many physics textbooks is not really correct. That is because the set of densities is not open in $L_1(dm).$ There are many
alternative proofs. Consider for example the work by Csiszar (1975) and Cherny and Maslov (2003).\\
The heuristics behind (\ref{sol2})and what comes next is the following. If in statement (ii) of Theorem (\ref{theo1}) we take $g(x)$ to be any member of the exponential family $\blambda \rightarrow
e^{-\sum_{k=0}^M \lambda_k h_k(x)},$ the inequality becomes
$$S_m(f) \leq \ln Z(\blambda) + \langle\blambda,\bd\rangle$$
which suggests that if we find a minimizer $\blambda^*$ such that the inequality becomes an equality, by Theorem (1) we conclude that (\ref{sol2}) is the desired solution. This dualization argument seems to have been first proposed in Mead and Papanicolau (1974), and is expounded in full rigor in Borwein and Lewis (2000). The vector $\mbox{\boldmath $\lambda$}^{*}$ can be found  minimizing the dual entropy:
\begin{equation}\label{dual}
\Sigma(\blambda, \bd) = \ln Z(\blambda) + \langle\blambda,\bd\rangle
\end{equation}
\noindent where $\bd$ is the $M-$vector with components $d_k,$ and obviously, the dependence on $\mbox{\boldmath $\alpha$}$ is through $\bd.$ We add that technically speaking, the minimization of $\Sigma(\mbox{\boldmath $\lambda$},\bd)$ is over the domain of $Z(\blambda)$ which is a convex set. In most applications it just is $\mathbb{R}^M.$ And for the record, we state the result of the duality result as
\begin{lemma}\label{dualeq}
$$S_m(f^*) = \Sigma(\blambda^*,\bd) = \ln Z(\blambda^*) + \langle\blambda^*,\bd\rangle$$
\end{lemma}
\subsection{Mathematical complement}
In this section we gather some results about $Z(\blambda)$ that we need below.
The following is well known. See Kullback (1968) for example.
\begin{proposition}\label{prop1}
With the notations introduced above, suppose that the matrix $\bC$  which we use to denote the covariance of $\bh(x)$ computed with respect the density $f^*,$ is strictly positive definite. Then\\
{\bf 1} The function $Z(\blambda)$ defined above is log-convex, that is, $\ln Z(\blambda)$ is convex.\\
{\bf 2} $Z(\blambda)$ is continuously differentiable as many times as we need. \\
{\bf 3} If we set $\phi(\blambda) = -\nabla_{\blambda}\ln Z(\blambda),$ then $\blambda = \phi^{-1}(\bd)$ is continuously differentiable in $\bd.$\\
{\bf 4} The Jacobian $\bD$ of $\phi^{-1}$ at $\bd$ equals the (negative) the covariance matrix of $\bh(x)$ computed with respect to $f^*(x).$
\end{proposition}
The first two assertions are proved in Kullback's book. The third drops out form the inverse function theorem in calculus. See Fleming (1987), and the last one follows from the fact that the Jacobian of $\phi^{-1}$ equals the negative of the inverse of the Hessian matrix of $\ln Z(\blambda),$ which is (minus) the covariance matrix $\bC.$
As a simple consequence of item ({\bf 4}) in Proposition(\ref{prop1}) we have the simple, but relevant for the next section
\begin{theorem}\label{theo2}
With the notations introduced above we have, with $\delta\bd=\hat{\bd}_N-\bd.$
 The change in $\blambda^*$ as $\bd\rightarrow \bd+\delta\bd$  up to terms $o(\delta\bd),$ is given by
$$\delta\blambda = \bD\delta\bd$$
\noindent and more importantly, using (\ref{sol2}), and again,  up to terms $o(\delta\bd),$
\begin{equation}\label{chan1}
\hat{f}^*_N(x) - f^*(x) = -f^*(x)\langle(\bh(x) -\bd),\bD(\delta\bd) \rangle.
\end{equation}
\end{theorem}

\section{Sample dependence}
Throughout this section, we shall consider a sample $\{X_1,...,X_N\}$ of size $N$ of the random variable $X.$ Here we shall realte the fuctuations of $\hat{d}_k = \frac{1}{N}\sum_{n=1}^N h_k(X_n)$ around its mean, to the fulctuations of the density. The following is well known
\begin{theorem}\label{lln}
Suppose that $\bh$ integrable, with mean $\bd$ and covariance mattix $\bSigma(h).$ Then,  for each $N,$ (\ref{estim1}) is an unbiased estimator of $\bd$ and 
$$\hat{d}_k = \frac{1}{N}\sum_{n=1}^N H_k(X_n) \rightarrow \bd\;\;\;\mbox{a.s.}\;\;\;\mathbb{P},\;\;\;\mbox{when}\;\;N\rightarrow\infty$$
\end{theorem}

 Let us begin with the simple
\begin{proposition}\label{sd1}
Define the empirical moments as in (\ref{estim1}. Denote by $\hat{\blambda}^*_N$ the Lagrange multiplier determined as explained in Section 2. Then, as $N \rightarrow\infty$, $\hat{\bd}_N\rightarrow \bd$ and therefore $\hat{\blambda}^*_N \rightarrow \blambda^*$.\\
If $\hat{f}^*_N$ and $f^*$ are the maxentropic densities given by (\ref{sol2}), corresponding, respectively, to $\hat{\bd}$ and $\bd,$ then $\hat{f}^*_N \rightarrow f^*$ pointwise.
\end{proposition}
The next result concerns the convergence of $\hat{f}^*_N$ to $f^*$ in $L_1(dm).$
\begin{theorem}\label{sd2}
With the notations introduced above, we have
$$\hat{f}^*_N \stackrel{L_1(dm)}{\longrightarrow} f^*\qquad \mbox{as}\qquad N\rightarrow \infty.$$
\end{theorem}
\begin{proof}
The proof is an easy consequence of the continuous dependence of $\Sigma(\blambda,\bd)$ on its arguments, of the identity (\ref{dualeq}) and item (iii) in Theorem (\ref{theo1}) with $\hat{f}_N^*$ playing the role of $f$ and $f^*$ playing the role of $g.$ In this case $-S_m(\hat{f}_M^*,f^*)$ happens to be
$$
-S_m(\hat{f}_N^*,f^*) = \langle\blambda^*,\hat{\bd}_N\rangle - \langle\blambda_N^*,\hat{\bd}_N\rangle + \ln Z(\blambda^*_N) - \ln Z(\blambda^*),$$
which, as mentioned, tends to $0$ as $N\rightarrow\infty.$
\end{proof}

To continue, consider
\begin{theorem}\label{sd3}
With the notations introduced above, we have\\
{\bf 1} $\hat{f}^*_N$ is an unbiased estimator of $f^*.$\\
{\bf 2} For any bounded, Borel measurable function $g(x),$
$$\begin{array}{l}
\left( \int g(x)\hat{f}_N^*(x)dm(x) - \int g(x)f^*(x)dm(x))\right)^2 \\
\leq \|\int\bh(x)g(x)f^*(x)-\bd\int g(x)f^*(x)\|^2\langle \delta\bd,\bD\delta\bd\rangle.
\end{array}
$$
\end{theorem}
The proof follows easily from (\ref{chan1}). In that inequality, the Cauchy-Schwartz inequality in $\mathbb{R}^K$ is used. Also, taking limits as $N\rightarrow\infty$ in (\ref{chan1}), we obtain another proof of the convergence of $\hat{f}_N^*$ to $f^*.$\\
What is interesting about ({\bf 2}) in Theorem (\ref{sd3}), is the possibility of combining it with Chebyshev's inequality to obtain rates of convergence. It is not hard to verify that
$$\mathbb{P}(\|D^{1/2}(\hat{\bd}_N -\bd)\| > a) \leq \frac{tr(\bSigma_N\bD)}{a^2}= \frac{tr(\bD\bSigma)}{Na^2},$$
\noindent where $\|\cdot\|$ is the Euclidean norm in $\mathbb{R}^K,$ and
$$\bSigma_N = E_{\mathbb{P}}[\delta\bd\delta\bd^t] = \frac{1}{N}E_\mathbb{P}[\left(\bh(X)-\bd\right)\left(\bh(X)-\bd\right)^t] = \frac{1}{N}\bSigma(\bh).$$
\begin{corollary}\label{cor1}
With the notations introduced in Theorem (\ref{sd3}) and two lines above,
$$\begin{array}{l}
\mathbb{P}\left(|\int g(x)\hat{f}_N^*(x)dm(x) - \int g(x)f^*(x)dm(x)| \leq a\right)\\
\geq 1 - \|\int\bh(x)g(x)f^*(x)-\bd\int g(x)f^*(x)\|^2tr(\bD\bSigma(\bh))/Na^2.
\end{array}$$
\end{corollary}
{\bf Comment} If we take $g=I_A$, we obtain a simple estimate of the speed of decay of
$|\int_A\hat{f}_N^*(x)dm(x)-\int_Af^*(x)dm(x)|$ to zero, or of the speed of convergence of $\int_A\hat{f}_N^*(x)dm(x)$ to $\int_Af^*(x)dm(x)$ if you prefer.
As far as the fluctuations around the mean, consider the following two possibilities.
\begin{theorem}\label{sd4}
With the notations introduced in Theorem (\ref{theo2}) and in the identity  (\ref{chan1}), we have
$$\sqrt{N}\left(\hat{f}^*_N(x) - f^*(x)\right) = -f^*(x)\langle(\bh(x) -\bd),\bD\left(\sqrt{N}(\delta\bd) \right)\rangle  \sim  N(0,\sigma^2(x))$$
in law as $N\rightarrow \infty.$ Also, for any bounded, Borel measurable $g(x)$
$$\begin{array}{l}
\sqrt{N}\left(\int g(x)\hat{f}_N^*(x)dm(x) - \int g(x)f^*(x)dm(x)\right) \\
= \langle\left(\int\bh(x)g(x)f^*(x)-\bd\int g(x)f^*(x)\right),\bD\left(\sqrt{N}(\delta\bd)\right)\rangle \sim N(0,\sigma^2(g)).
\end{array}$$
in law. Above, $\sigma^2(x) = \langle\bv(x),\bSigma(\bh)\bv(x)\rangle$, $\sigma^2(g) =  \langle\bv(g),\bSigma(\bh)\bv(g)\rangle$, where $\bv(x)=f^*(x)\bD^{1/2}\left(\bh(x)-\bd\right)$ and $\bv(g) = \int g(x)\bv(x)dm(x).$
\end{theorem}
The proof of the assertions is standard. It involves applying the central limit theorem to the vector variable $\sqrt{N}\left(\hat{\bh}_N - \bd\right).$

  \end{document}